\documentclass[11pt]{article}

\usepackage[T1]{fontenc}
\usepackage{lmodern}
\usepackage{amsmath,amssymb,amsthm,mathtools}
\usepackage[margin=1in]{geometry}
\usepackage{microtype}
\usepackage[colorlinks=true,linkcolor=blue,citecolor=blue,urlcolor=blue]{hyperref}
\hypersetup{
  pdftitle={Affine Spherical Mass and a Gaussian Kinematic Formula for Tropical Two-Fans},
  pdfauthor={Nikita Kalinin},
  pdfkeywords={tropical geometry, stable intersection, balanced fans,
    lattice projections, integral geometry, affine isotropic position}
}

\theoremstyle{plain}
\newtheorem{theorem}{Theorem}[section]
\newtheorem{proposition}[theorem]{Proposition}
\newtheorem{lemma}[theorem]{Lemma}
\newtheorem{corollary}[theorem]{Corollary}
\theoremstyle{definition}
\newtheorem{example}[theorem]{Example}
\newtheorem{definition}[theorem]{Definition}
\theoremstyle{remark}

\newcommand{\R}{\mathbb R}
\newcommand{\Z}{\mathbb Z}
\newcommand{\E}{\mathbb E}
\newcommand{\tr}{\operatorname{tr}}
\newcommand{\Span}{\operatorname{span}}
\newcommand{\relint}{\operatorname{relint}}
\newcommand{\Sym}{\operatorname{Sym}}
\newcommand{\Id}{\mathrm I}
\newcommand{\covol}{\operatorname{covol}}
\newcommand{\rank}{\operatorname{rank}}
\newcommand{\norm}[1]{\lVert#1\rVert}
\newcommand{\ip}[2]{\langle#1,#2\rangle}
\newcommand{\Gr}{\operatorname{Gr}}

\title{Affine Spherical Mass and a Gaussian Kinematic Formula for Tropical
Two-Fans}
\author{Nikita Kalinin\\[0.5ex]
  {\small Faculty of Mathematics, Technion--Israel Institute of Technology,}\\
  {\small Haifa 3200003, Israel}\\
  {\small Guangdong Technion--Israel Institute of Technology,}\\
  {\small 241 Daxue Road, Shantou, Guangdong 515063, China}\\[0.5ex]
  {\small \href{mailto:nikita.kalin@technion.ac.il}
  {\texttt{nikita.kalin@technion.ac.il}}}\\
  {\small \href{https://orcid.org/0000-0002-1613-5175}
  {ORCID: 0000-0002-1613-5175}}}
\date{}

\begin{document}
\maketitle

\begin{abstract}
Let \(F\) be an effective balanced rational two-dimensional fan in a
rank-four lattice.  We introduce an affine spherical mass
\[
 M_*(F)=\inf_{\covol_g(N)=1}M_g(F)
\]
and prove the two-sided comparison
\[
 \sqrt{\frac83}\,m(F)\leq M_*(F)\leq
 3\sqrt{2\,q(F)}.
\]
Here \(q(F)=\deg(F\cdot F)\) is stable self-intersection and \(m(F)\)
is the least degree of \(F\) under a primitive rank-two lattice quotient.
For a surjection \(A\), one has
\(\deg(A_*F)=\deg(F\cdot[\ker A_\R])\); hence \(m(F)\) is equivalently the
least stable intersection degree with a complete weight-one rational
two-plane.
In particular,
\[
 q(F)\geq\frac4{27}m(F)^2.
\]
The proof combines an exact Gaussian--kinematic formula for mixed stable
intersection with an affine normalization principle for stationary spherical
links.  In dimension four, the identification
\(\Gr^+(2,4)\simeq S^2\times S^2\) converts ambient isotropy into a sharp
second-moment transversality estimate.  We prove affine covariance for
\(M_*\), derive equality and stability conditions for the transversality
estimate, and give examples where ambient isotropy coexists with zero mixed
intersection, as well as a three-fan in six-space with zero self-intersection.
Determining the optimal numerical constant remains open.
\end{abstract}

\medskip
\noindent\textbf{Keywords.}
Tropical geometry; stable intersection; balanced fans; lattice projections;
integral geometry; affine isotropic position.
\medskip

\section{Introduction}

The planar prototype has a direct tropical formulation.  Rotating the
weighted primitive ray generators of an integral effective balanced one-fan
in a rank-two lattice by a quarter-turn gives the oriented edge vectors of a
possibly degenerate convex lattice polygon \(P\).
With the ordered-pair convention for stable intersection,
\[
 \deg(F\cdot F)=2\operatorname{area}(P),\qquad
 \min_{A:\Z^2\twoheadrightarrow\Z}\deg(A_*F)=\omega(P),
\]
where \(\omega(P)\) is the lattice width.  The sharp inequality of Fejes
T\'oth and Makai
\[
 \operatorname{area}(P)\geq\frac38\omega(P)^2
\]
therefore says \(\deg(F\cdot F)\geq\frac34\omega(P)^2\) in this
normalization \cite{FejesTothMakai}.  Equivalently, it bounds the
self-intersection of a one-dimensional fan in terms of its smallest integral
projection.

The constant \(3/8\) also appears in several related settings: in the region
of influence of a singular point in a Newton subdivision
\cite{KalininNewton}, as the density constant for non-separable lattices of
convex plates, and as a width bound in the enumeration of lattice polygons.
In particular, area at most \(A\) implies
\(\omega(P)\leq\sqrt{8A/3}\), providing a width cutoff in the enumeration of
lattice polygons of area at most \(A\), up to affine unimodular equivalence
\cite{BaranyPach,CoolsLemmens,Balletti}.
Related lattice inequalities enter systolic problems for Finsler two-tori
and symplectic capacities of optical hypersurfaces
\cite{AlvarezPaivaBalacheffTzanev}.  Aliev proved a three-dimensional
convex-geometric analogue involving exact reverse-isodiametric and
lattice-width bounds \cite{AlievThreeSpace}.

This paper studies the next tropical case: an effective balanced two-fan
\(F\) in a four-dimensional lattice \(N\).  Put
\[
 q(F)=\deg(F\cdot F),\qquad
 m(F)=\min_{A:N\twoheadrightarrow\Z^2}\deg(A_*F).
\]
For \(K_A=\ker(A\otimes_\Z\R)\), endowed with its primitive complete-plane
weight,
\[
 \deg(A_*F)=\deg(F\cdot[K_A]).
\]
Conversely, every rational two-plane in \(N_\R\), equivalently every plane
spanned over \(\R\) by two independent vectors of \(N\), occurs in this
way.  Thus \(m(F)\) is the least stable intersection degree with such a
weight-one plane \cite[Section~2]{JensenYu}.
Question~12 of Kalinin's thesis asks whether
\(q(F)\geq c\,m(F)^2\) for a universal \(c>0\) \cite{KalininThesis}.  Hept
and Theobald study a related projection problem---self-crossings of tropical
varieties---in a different projection dimension \cite{HeptTheobald}.

The affine spherical mass is the intermediate invariant linking \(m(F)\) and
\(q(F)\).  The central statement is
\begin{equation}\label{eq:intro-sandwich}
 \boxed{\sqrt{\frac83}\,m(F)\leq M_*(F)
 \leq3\sqrt{2q(F)}.}
\end{equation}
The proof consists of four steps.
\begin{enumerate}
\item For every volume-one metric \(g\), the spherical link of \(F\) is a
stationary weighted geodesic network of mass \(M_g(F)\).
\item Stable intersection admits an exact mixed Gaussian formula
\[
 \deg(F\cdot G)=M_g(F)M_g(G)
 \E[B(z,z')^2J(x\cdot y)].
\]
This follows by combining a Gaussian coarea calculation
\cite{Santalo,HowardKinematic,AmelunxenLotz} with the tropical
fan-displacement rule \cite{FultonSturmfels,JensenYu}.
\item When \(M_*(F)>0\), a first-variation argument produces an
almost-minimizing sequence whose tangent covariance tends to the isotropic
matrix.  This parallels Petty's characterization of minimal surface-area position
\cite{PettySurfaceArea,GiannopoulosPapadimitrakis}.
\item The special geometry of \(\Gr^+(2,4)\) yields an exact defect formula
for average transversality, while Rankin's value
\(\gamma_{4,2}=3/2\) bounds a projection degree.  These estimates give the
two sides of \eqref{eq:intro-sandwich}.
\end{enumerate}

The role of self-intersection in four dimensions is essential.  For two
distinct fans, even a common mass-minimizing isotropic metric may leave every
cross-pair of facet planes with a common line, and hence give zero mixed
intersection.  Likewise, a three-fan in \(\R^6\) may have positive affine
mass, exact ambient isotropy, and zero self-intersection.  Positive average
transversality therefore depends on the exceptional geometry of two-planes
in four-space.

\section{Conventions and the spherical mass}

Let \(N\simeq\Z^4\), \(V=N\otimes_\Z\R\), and
\[
 F=\sum_\sigma w_\sigma[\sigma]
\]
be an effective balanced rational two-fan.  We use ordered cone pairs in
the fan-displacement product.  Thus the sum of two complementary primitive
planes of weight one has self-intersection \(2\).

For a lattice quotient \(A:N\twoheadrightarrow\Z^2\), write
\[
 A_*F=d_A(F)[\R^2].
\]
The degrees \(d_A(F)\) are nonnegative integers for integral \(F\), and
\[
 m(F)=\min_A d_A(F).
\]
If the weights have common denominator \(D\), then every \(d_A(F)\) lies in
\(D^{-1}\Z_{\geq0}\), so the minimum exists.  The invariants \(m\) and \(q\)
are homogeneous of degrees one and two, respectively.

After a rational subdivision we may suppose all facets pointed.  Fix a
Euclidean metric \(g\) with \(\covol_g(N)=1\).  For a facet \(\sigma\), let
\[
 N_\sigma=N\cap\Span_\R(\sigma),\qquad
 p_\sigma\in\bigwedge^2N_\sigma
\]
be a primitive generator up to sign, and put
\(I_\sigma=\sigma\cap S_g^3\).  On this great-circle arc define
\begin{equation}\label{eq:linkmeasure}
 d\mu_{F,g}=
 \frac{w_\sigma\norm{p_\sigma}_g}{2\pi}\,ds,\qquad
 M_g(F)=\mu_{F,g}(S_g^3).
\end{equation}
Parametrize at unit speed and write \(x=x(s)\), \(t=x'(s)\), with orientation
chosen so that
\begin{equation}\label{eq:unitbivector}
 z=x\wedge t=\frac{p_\sigma}{\norm{p_\sigma}_g}.
\end{equation}
When \(M_g(F)>0\), the normalized flag law is
\(\nu_{F,g}=\mu_{F,g}/M_g(F)\).

At a link vertex \(x=a/\norm a\), project the tropical balancing equation at
the ray \(\R_{\geq0}a\) orthogonally to \(a\).  This gives the tension equation
\begin{equation}\label{eq:tension}
 \sum_{\sigma\ni a}w_\sigma\norm{p_\sigma}\,
 t_{\sigma,\mathrm{out}}=0.
\end{equation}
Thus the link is a stationary geodesic network in the standard sense of
first variation \cite{Allard,AllardAlmgren}.  Since \(x'=t\) and \(t'=-x\)
on a great circle, integrating
\((x\otimes t)'=t\otimes t-x\otimes x\) and using
\eqref{eq:tension} yields
\begin{equation}\label{eq:stationarity}
 \int t\otimes t\,d\mu_{F,g}
 =\int x\otimes x\,d\mu_{F,g}.
\end{equation}
Taking expectation with respect to \(\nu_{F,g}\), set
\begin{equation}\label{eq:T}
 T_g=\E_g(tt^T)=\E_g(xx^T),\qquad
 \tr T_g=1.
\end{equation}
All unlabeled matrix norms below are Hilbert--Schmidt norms for the ambient
metric.

\section{The affine mass}

Combine the numerical weight and primitive lattice bivector into the facet
density
\[
 \xi_\sigma=w_\sigma p_\sigma.
\]
For \(h\in GL(V)\), send \((\sigma,\xi_\sigma)\) to
\((h\sigma,(\bigwedge^2h)\xi_\sigma)\), and denote the resulting weighted fan
by \(h_\#F\).  Directly from \eqref{eq:linkmeasure},
\begin{equation}\label{eq:masscovariance}
 M_g(h_\#F)=M_{h^*g}(F).
\end{equation}
This defines an action on real weighted conical densities and preserves
integrality for lattice automorphisms.

\begin{definition}\label{def:affinemass}
The affine spherical mass of \(F\), relative to the lattice volume, is
\[
 M_*(F)=\inf_{\covol_g(N)=1}M_g(F).
\]
\end{definition}

\begin{proposition}[Affine and convex-cone properties]\label{prop:mass-properties}
The following statements hold.
\begin{enumerate}
\item Both \(M_g\) and \(M_*\) are invariant under subdivision, and
\(M_*(\lambda F)=\lambda M_*(F)\) for \(\lambda\geq0\).
\item For effective fans,
\[
 M_*(F+G)\geq M_*(F)+M_*(G).
\]
\item Keeping the ambient lattice volume fixed,
\[
 M_*(h_\#F)=|\det h|^{1/2}M_*(F).
\]
In particular \(M_*\) is \(SL(V)\)-invariant.  If the lattice is transported
to \(hN\) as well, then \(M_{*,hN}(h_\#F)=M_{*,N}(F)\).
\item If \(N'\subset N\) has index \(I\) and the geometric densities
\(\xi_\sigma\) are kept fixed, then
\[
 M_{*,N'}(F)=I^{-1/2}M_{*,N}(F).
\]
\end{enumerate}
\end{proposition}

\begin{proof}
The new spherical arcs in a subdivision partition the old arcs, so their
masses add.  Homogeneity is immediate.  After a common refinement,
\(M_g(F+G)=M_g(F)+M_g(G)\); taking infima proves superadditivity.

For affine covariance, let \(g\) have lattice covolume one and set
\[
 \widehat g=|\det h|^{-1/2}h^*g.
\]
Then \(\covol_{\widehat g}(N)=1\), while uniform metric scaling and
\eqref{eq:masscovariance} give
\[
 M_g(h_\#F)=M_{h^*g}(F)=|\det h|^{1/2}M_{\widehat g}(F).
\]
The map \(g\mapsto\widehat g\) is a bijection of the volume-one metric space,
so taking infima proves the third assertion.  If \(N'\subset N\) has index
\(I\), then
\(I^{-1/2}g\) has \(N'\)-covolume one whenever \(g\) has \(N\)-covolume one,
and its link mass is \(I^{-1/2}M_g(F)\).  This proves the fourth assertion.
\end{proof}

The finite-index statement tracks the geometric density \(\xi_\sigma\).
Retaining the raw numerical weights makes the primitive bivector of each
facet change by its individual index
\[
 [N\cap\Span(\sigma):N'\cap\Span(\sigma)],
\]
so the transformation law is determined facet by facet.
Weight scaling and ambient dilation act differently.  Multiplying
all weights by \(c\geq0\) multiplies \(M_*\) by \(c\).  The pushed-density
action of the scalar map \(h=\lambda\Id\) gives
\[
 M_*((\lambda\Id)_\#F)=|\lambda|^2M_*(F).
\]
Set-theoretically a fan is unchanged by positive scalar dilation; the latter
formula concerns its transported two-dimensional density.

Thus \(M_*\) is a homogeneous concave functional on the cone of effective
fans.
Superadditivity can be strict.  This already occurs for the complementary
coordinate planes
\[
 P=\Span(e_1,e_2),\qquad Q=\Span(e_3,e_4).
\]
Each plane individually has affine mass zero.  Their sum satisfies
\[
 M_g(P+Q)=\norm{e_1\wedge e_2}_g+
          \norm{e_3\wedge e_4}_g\geq2
\]
for every volume-one metric, with equality in the standard metric.  Hence
\(M_*(P+Q)=2\).  Moreover,
\[
 q(P+Q)=2,\qquad m(P+Q)=0,
\]
since the quotient with kernel \(\Span(e_2,e_4)\) has rank at most one on
each component.  Thus positive affine mass can coexist with zero projection
minimum.  Every fan with \(m(F)>0\) has \(M_*(F)>0\).

Affine covariance also characterizes vanishing:
\begin{equation}\label{eq:affine-collapse}
 M_*(F)=0\quad\Longleftrightarrow\quad
 \text{there exist \(h_k\in SL(V)\) with }
 M_{g_0}(h_{k\#}F)\longrightarrow0
\end{equation}
for some, and hence every, reference metric \(g_0\).  Thus \(M_*(F)=0\)
means affine collapse of the weighted link.  This occurs when the support
lies in a proper linear subspace, or when all facet planes contain a common
line:
scale that line by \(\varepsilon^3\) and a complementary three-space by
\(\varepsilon^{-1}\); every density bivector then scales by
\(\varepsilon^2\).  These are sufficient conditions; a full flag-theoretic
or Hilbert--Mumford-type criterion remains open.

\section{An exact mixed Gaussian formula}

Let \(F=\sum_\sigma w_\sigma[\sigma]\) and
\(G=\sum_\tau v_\tau[\tau]\) be effective balanced rational two-fans in
\(N_\R\).  For nonzero \(F,G\), equip their links with the normalized laws
associated with the same volume-one metric.  For flags
\[
 z=x\wedge t,\qquad z'=y\wedge u,
\]
put
\[
 B(z,z')=\det_g(x,t,y,u),\qquad c=x\cdot y,
\]
and define
\begin{equation}\label{eq:J}
 J(c)=\int_{r,\rho>0}r\rho
 \exp\left(-\frac{r^2+\rho^2-2cr\rho}{2}\right)\,dr\,d\rho,
 \qquad -1\leq c<1.
\end{equation}
Whenever the two facet planes share a line, the expression
\(B(z,z')^2J(x\cdot y)\) is defined to be zero as a whole.  This convention
also covers the formal endpoint \(x\cdot y=1\), where \(B=0\).

\begin{lemma}[Radial Jacobians]\label{lem:radial-jacobians}
Let \(x=x(s)\) and \(y=y(a)\) be unit-speed spherical arcs, with
\(t=x'(s)\) and \(u=y'(a)\).
\begin{enumerate}
\item The map
\[
 (r,s,\rho,a)\longmapsto rx(s)-\rho y(a)
\]
has absolute Jacobian
\(r\rho|\det(x,t,y,u)|\).
\item If the restriction of \(A\) to \(\Span(x,t)\) is an isomorphism onto
an oriented Euclidean plane, then
\[
 \left\lvert\frac{d}{ds}\frac{Ax}{\norm{Ax}}\right\rvert
 =\frac{|\det(Ax,At)|}{\norm{Ax}^2}.
\]
\end{enumerate}
\end{lemma}

\begin{proof}
For the first assertion, take the determinant of the derivative columns
\(x,rt,-y,-\rho u\).  The second is the planar polar-coordinate identity
\(|d\theta/ds|=|\det(v,v')|/\norm v^2\), applied to \(v=Ax\).
\end{proof}

\begin{theorem}[Mixed Gaussian--kinematic formula]\label{thm:mixed-gaussian}
For \(F\neq0\), \(G\neq0\), and every volume-one metric \(g\),
\begin{equation}\label{eq:mixed-gaussian}
 \boxed{\deg(F\cdot G)=M_g(F)M_g(G)
 \E_{\nu_{F,g}\otimes\nu_{G,g}}
 \left[B(z,z')^2J(x\cdot y)\right].}
\end{equation}
Moreover, for \(-1\leq c<1\),
\begin{equation}\label{eq:Jlower}
 J(c)\geq J(-1)=\frac13.
\end{equation}
For \(F=0\) or \(G=0\), equation \eqref{eq:mixed-gaussian} carries the
convention that both sides equal zero.
\end{theorem}

\begin{proof}
For transverse facets \(\sigma,\tau\), the map
\[
 (r,s,\rho,a)\longmapsto rx(s)-\rho y(a),\qquad r,\rho>0,
\]
parametrizes \(\relint(\sigma-\tau)\) and, by
Lemma~\ref{lem:radial-jacobians}, has absolute Jacobian \(r\rho|B|\).
For a fixed pair of oriented facet planes, the unit bivectors \(z,z'\), and
hence \(B\), are constant along their link arcs.  Therefore a standard
Gaussian displacement \(\delta\)
belongs to this difference cone with probability
\[
 \frac{|B|}{(2\pi)^2}
 \int_{I_\sigma\times I_\tau}J(x\cdot y)\,ds\,da.
\]
Because the lattice covolume is one, the index of a transverse pair is
\[
 [N:N_\sigma+N_\tau]
 =\norm{p_\sigma}\norm{p_\tau}|B|.
\]
The fan-displacement rule says that, for generic \(\delta\),
\[
 \deg(F\cdot G)=
 \sum_{\substack{\sigma\in F,\ \tau\in G\\
 \delta\in\relint(\sigma-\tau)}}
 w_\sigma v_\tau[N:N_\sigma+N_\tau].
\]
The displacement rule is formulated for small generic \(\delta\).  Since
each difference cone \(\sigma-\tau\) is conical, the contributing pair list
is invariant under positive rescaling; hence the formula may be averaged over
an unrestricted Gaussian.  Nontransverse difference cones have Gaussian
measure zero.  For \(F=G\), ordered facet pairs agree with the ordered
fan-displacement convention.  Combining the Gaussian probability with the
lattice index and \eqref{eq:linkmeasure} proves
\eqref{eq:mixed-gaussian} \cite{FultonSturmfels,JensenYu}.

For \(r,\rho>0\), the integrand in \eqref{eq:J} is pointwise strictly
increasing in \(c\), through the factor \(e^{cr\rho}\).  Hence \(J\) is
strictly increasing on \([-1,1)\).  With \(s=r+\rho\) and \(t=r/s\),
\[
 J(-1)=\int_{r,\rho>0}r\rho e^{-(r+\rho)^2/2}\,dr\,d\rho
 =\frac16\int_0^\infty s^3e^{-s^2/2}\,ds=\frac13.
\]
\end{proof}

For a metric of arbitrary lattice covolume \(D\), the right side of
\eqref{eq:mixed-gaussian} carries the additional factor \(D^{-1}\).
The covolume-one normalization is therefore essential.

\section{Affine normalization}

The affine-normalization step for \(F\) uses only positivity.  Balancing
enters later, when tangent isotropy is converted into plane isotropy.

Let \(\mathcal P=SL_4(\R)/SO(4)\) be the space of volume-one metrics.  We
normalize its invariant Riemannian metric so that, in a \(g\)-orthonormal
frame, the curve \(g_\eta=(e^{\eta H})^*g\) has initial speed
\(\norm H_{\mathrm{HS}}\) for symmetric trace-free \(H\).

\begin{lemma}[Approximate critical points]\label{lem:approx-critical}
Let \(\mathcal X\) be a proper Hadamard manifold.  If
\(f:\mathcal X\to(0,\infty)\) is \(C^1\), there are \(x_k\in\mathcal X\)
such that
\[
 f(x_k)\longrightarrow\inf_{\mathcal X}f,\qquad
 \norm{\nabla f(x_k)}\longrightarrow0.
\]
\end{lemma}

\begin{proof}
Fix \(x_0\in\mathcal X\).  For \(\varepsilon>0\), the coercive function
\[
 f(x)+\varepsilon d(x,x_0)^2
\]
attains a minimum \(x_\varepsilon\).  Comparison with \(x_0\), followed by
the Euler equation, gives
\[
 \norm{\nabla f(x_\varepsilon)}
 =2\varepsilon d(x_\varepsilon,x_0)
 \leq2\sqrt{\varepsilon f(x_0)}.
\]
For each fixed \(y\in\mathcal X\), minimality also gives
\[
 \inf f\leq f(x_\varepsilon)
 \leq f(y)+\varepsilon d(y,x_0)^2.
\]
Letting \(\varepsilon\downarrow0\) and then taking the infimum over \(y\)
proves the assertion.
\end{proof}

For the fixed pointed subdivision, \(g\mapsto M_g(F)\) is smooth: each facet
contributes a bivector norm times a spherical arc length, and both depend
smoothly on \(g\).  Hence Lemma~\ref{lem:approx-critical} applies to the mass.

\begin{theorem}[Almost-isotropic affine position]\label{thm:isotropic}
If \(M_*(F)>0\), there are volume-one metrics \(g_k\) such that
\begin{equation}\label{eq:isotropic}
 M_{g_k}(F)\longrightarrow M_*(F),\qquad
 \norm{T_{g_k}-\frac14\Id_4}_{\mathrm{HS},g_k}\longrightarrow0.
\end{equation}
If the infimum is attained, every minimizing metric is tangent-isotropic.
\end{theorem}

\begin{proof}
Represent a metric variation by \(g_\eta=h_\eta^*g\), where
\(h_\eta=e^{\eta H}\) and \(H\) is symmetric and trace-free in a
\(g\)-orthonormal frame.  Identify the \(g_\eta\)-sphere with the \(g\)-sphere
by radial projection.  The mass element at \(z=x\wedge t\) is then multiplied
by
\[
 \frac{\norm{\bigwedge^2h_\eta z}^{\,2}}
      {\norm{h_\eta x}^{\,2}}.
\]
Differentiating its logarithm at \(\eta=0\) gives
\(2\ip{Ht}{t}\); the radial terms cancel.  Thus
\begin{equation}\label{eq:firstvariation}
 DM_g(H)=2M_g(F)\tr(HT_g).
\end{equation}
With the normalization above,
\[
 \nabla\log M_g=2\left(T_g-\frac14\Id_4\right).
\]
Apply Lemma~\ref{lem:approx-critical} to \(f(g)=M_g(F)\).  Since the limiting
mass is positive, the displayed logarithmic gradient forces
\(\norm{T_{g_k}-\Id_4/4}_{\mathrm{HS},g_k}\to0\).
At an actual minimizer the same first variation vanishes identically.
\end{proof}

Along the sequence, average tangent mass becomes equidistributed among
ambient directions, although individual facets may remain concentrated.
For a balanced fan, \eqref{eq:stationarity} also gives, for
\(\Pi_z=xx^T+tt^T\),
\begin{equation}\label{eq:planeisotropy}
 \E_g\Pi_z=2T_g.
\end{equation}
The first-variation argument uses positivity; this last identity is the point
at which balancing enters.

\section{The four-dimensional Hodge defect}

We now convert plane isotropy into a quantitative lower bound for average
transversality.

Fix an orientation.  Split
\[
 \bigwedge^2V=\bigwedge^2_+V\oplus\bigwedge^2_-V.
\]
Using the unit spheres in
\(\bigwedge^2_+V\simeq\R^3\) and \(\bigwedge^2_-V\simeq\R^3\), every unit
simple bivector has a unique representation
\begin{equation}\label{eq:XY}
 z=\frac{(X,Y)}{\sqrt2},\qquad X,Y\in S^2.
\end{equation}
There is an \(SO(4)\)-equivariant linear isometry
\[
 \Phi:\R^3\otimes\R^3\longrightarrow\Sym_0(\R^4)
\]
such that
\begin{equation}\label{eq:Phi}
 \Pi_z-\frac12\Id_4=\Phi(XY^T).
\end{equation}
This is the tensor form of the standard identification
\(\Gr^+(2,4)\simeq S^2\times S^2\).

To fix the normalization, define, for \(\omega\in\bigwedge^2V\), \(K_\omega\) by
\(\ip{K_\omega v}{w}=\ip{\omega}{v\wedge w}\).  Unit self-dual and
anti-self-dual forms satisfy
\[
 K_X^2=K_Y^2=-\frac12\Id_4,\qquad K_XK_Y=K_YK_X.
\]
Since \(K_z=(K_X+K_Y)/\sqrt2\) and \(K_z^2=-\Pi_z\),
\[
 \Pi_z-\frac12\Id_4=-K_XK_Y.
\]
Thus \(\Phi(X\otimes Y)=-K_XK_Y\).  This map is \(SO(4)\)-equivariant, and
\(\bigwedge^2_+V\otimes\bigwedge^2_-V\) is irreducible.  Its scale is fixed
by taking
\[
 X=(e_{12}+e_{34})/\sqrt2,\qquad
 Y=(e_{12}-e_{34})/\sqrt2,
\]
for which \(-K_XK_Y=\operatorname{diag}(1,1,-1,-1)/2\) has
Hilbert--Schmidt norm one.  Hence \(\Phi\) is an isometry and
\eqref{eq:Phi} follows.

For a probability law on oriented two-planes set
\begin{equation}\label{eq:PQR}
 P=\E(XX^T),\qquad Q=\E(YY^T),\qquad R=\E(XY^T).
\end{equation}
The matrices \(P,Q\) are positive semidefinite and have trace one.
For the stationary link law of a balanced fan, taking expectations in
\eqref{eq:Phi} and using \eqref{eq:planeisotropy} gives
\begin{equation}\label{eq:Rdefect}
 \norm R=2\norm{T_g-\frac14\Id_4}.
\end{equation}

\begin{theorem}[Exact transversality defect]\label{thm:defect}
For two independent flags with the same law,
\begin{equation}\label{eq:Bmoment}
 \E B(z,z')^2
 =\frac14\left(\tr(P^2)+\tr(Q^2)-2\norm R^2\right).
\end{equation}
Equivalently,
\begin{equation}\label{eq:hodge-defect}
 \E B^2=\frac16+\frac14\left(
 \norm{P-\frac13\Id_3}^2+\norm{Q-\frac13\Id_3}^2
 -2\norm R^2\right).
\end{equation}
For a stationary link this becomes
\begin{equation}\label{eq:defect}
 \boxed{\E B^2=\frac16+
 \frac14\left(
 \norm{P-\frac13\Id_3}^2+\norm{Q-\frac13\Id_3}^2\right)
 -2\norm{T_g-\frac14\Id_4}^2.}
\end{equation}
In particular,
\[
 \E B^2\geq\frac16-2\norm{T_g-\frac14\Id_4}^2.
\]
For a stationary link with \(T_g=\Id_4/4\), one has \(\E B^2\geq1/6\),
with equality exactly when
\[
 P=Q=\frac13\Id_3.
\]
\end{theorem}

\begin{proof}
For two oriented planes represented as in \eqref{eq:XY},
\begin{equation}\label{eq:BXY}
 B(z,z')=\frac{X\cdot X'-Y\cdot Y'}2.
\end{equation}
Expanding the square and using independence gives \eqref{eq:Bmoment}.  The
identity
\[
 \tr(P^2)=\frac13+\norm{P-\frac13\Id_3}^2
\]
and its analogue for \(Q\) give \eqref{eq:hodge-defect}.  Substituting
\eqref{eq:Rdefect} gives \eqref{eq:defect}, including the lower bound and its
equality criterion.
\end{proof}

For a stationary link, the defect identity gives a quantitative stability
statement.  If
\[
 \delta=\norm{T_g-\frac14\Id_4},\qquad
 \E B^2\leq\frac16+\eta,
\]
then
\begin{equation}\label{eq:defect-stability}
 \norm{P-\frac13\Id_3}^2+\norm{Q-\frac13\Id_3}^2
 \leq4\eta+8\delta^2.
\end{equation}
Thus, along an almost-isotropic sequence of stationary links, approaching the
lower transversality value \(1/6\) is equivalent to both Hodge marginals
becoming second-moment isotropic.

The constant \(1/6\) in the isotropic stationary-link inequality is sharp
among effective balanced rational fans.
Take the sum of the six complete coordinate planes
\[
 F_{\mathrm{coord}}=\sum_{1\leq i<j\leq4}
 [\Span(e_i,e_j)]
\]
in the standard metric.  Each direction occurs in three planes, so
\(T=\Id_4/4\).  Of the \(36\) ordered pairs of coordinate planes, exactly
six are complementary and have \(B^2=1\); the remaining thirty have
\(B=0\).  Consequently \(\E B^2=1/6\), so the value \(1/6\) is attained.

\section{Projection degrees and Rankin's theorem}

It remains to compare affine mass with the least lattice-quotient degree.

Let \(A:N\twoheadrightarrow\Z^2\), and let
\(\alpha\in\bigwedge^2N^*\) be a primitive generator of its saturated row
lattice.  Give the target a volume-one metric.  Tropical push-forward
multiplicities, followed by angular averaging in the target, give
\begin{equation}\label{eq:degree}
 d_A(F)=\frac1{2\pi}
 \sum_{\rank(A|_{\Span\sigma})=2}
 w_\sigma|\alpha(p_\sigma)|\,\operatorname{angle}(A\sigma).
\end{equation}
Choose the target metric so that the two nonzero singular values of \(A\)
are equal.  If \(S_A=A^*A\), then
\begin{equation}\label{eq:S}
 S_A=a\Pi_W,\qquad W=(\ker A)^{\perp_g},\qquad
 a=\norm\alpha_g.
\end{equation}
Thus \(W\subset V\) is the \(g\)-dual of the row plane in \(V^*\).
Indeed, if \(G\) is the Gram matrix of the two rows in the source dual
metric, the target Gram matrix \(aG^{-1}\) has determinant one and produces
\eqref{eq:S}.

On a link arc put
\[
 D=|\det(Ax,At)|=\frac{|\alpha(p_\sigma)|}{\norm{p_\sigma}}.
\]
Lemma~\ref{lem:radial-jacobians} gives angular speed
\(D/\norm{Ax}^2\), while
\[
 D^2\leq\norm{Ax}^2\norm{At}^2.
\]
Substituting this angular speed into \eqref{eq:degree} gives the exact identity
\begin{equation}\label{eq:projection-exact}
 d_A(F)=\int\frac{D^2}{\norm{Ax}^2}\,d\mu_{F,g},
\end{equation}
with zero contribution from rank-deficient facets.  Applying the displayed
pointwise inequality yields the basic estimate
\begin{equation}\label{eq:projection}
 \boxed{d_A(F)\leq M_g(F)\tr(S_AT_g).}
\end{equation}

Rankin's four-dimensional two-plane theorem says that every volume-one
rank-four Euclidean lattice contains independent \(u,v\) with
\begin{equation}\label{eq:Rankin}
 \norm{u\wedge v}^2\leq\gamma_{4,2}=\frac32
\end{equation}
\cite{Rankin,Watanabe,Sawatani}.  Apply this to the dual lattice and
saturate \(\Z u+\Z v\).  Saturation weakly decreases covolume, so the
primitive generator \(\alpha\) of the saturated exterior square satisfies
\[
 \norm\alpha\leq\sqrt{\frac32}.
\]
The saturated row lattice therefore defines a surjection
\(N\twoheadrightarrow\Z^2\).

Since \(S_A=a\Pi_W\) and \(\tr T_g=1\), \eqref{eq:projection} and Rankin's
bound give, for every volume-one metric,
\begin{equation}\label{eq:coarse}
 m(F)\leq\sqrt{\frac32}\,M_g(F).
\end{equation}
Now suppose \(M_*(F)>0\) and use Theorem~\ref{thm:isotropic}.  The Rankin
quotient may depend on \(k\); write \(W_k\subset V\) for the \(g_k\)-dual
of its row plane,
\(\alpha_k\) for its primitive row bivector, and
\(a_k=\norm{\alpha_k}_{g_k}\).  Rankin gives
\(a_k\leq\sqrt{3/2}\), while
\(\norm{\Pi_{W_k}}_{\mathrm{HS}}=\sqrt2\); therefore
\[
 \begin{aligned}
 m(F)&\leq a_kM_{g_k}(F)\tr(\Pi_{W_k}T_{g_k})\\
 &\leq\sqrt{\frac32}\,M_{g_k}(F)
 \left(\frac12+\sqrt2
 \norm{T_{g_k}-\frac14\Id_4}_{\mathrm{HS}}\right).
 \end{aligned}
\]
Letting \(k\to\infty\) proves
\begin{equation}\label{eq:massprojection}
 \boxed{m(F)\leq\sqrt{\frac38}\,M_*(F).}
\end{equation}
If \(M_*(F)=0\), the coarse estimate already gives \(m(F)=0\), so
\eqref{eq:massprojection} holds in every case.

\section{The mass sandwich and the universal bound}

\begin{theorem}[Affine mass comparison]\label{thm:mass-sandwich}
Every effective balanced rational two-fan in a rank-four lattice satisfies
\begin{equation}\label{eq:mass-sandwich}
 \boxed{\sqrt{\frac83}\,m(F)\leq M_*(F)
 \leq3\sqrt{2q(F)}.}
\end{equation}
Consequently,
\begin{equation}\label{eq:main}
 \boxed{q(F)\geq\frac4{27}m(F)^2.}
\end{equation}
\end{theorem}

\begin{proof}
The first inequality is \eqref{eq:massprojection}.  The second is immediate
when \(M_*(F)=0\).  For \(M_*(F)>0\), choose the almost-isotropic metrics of
Theorem~\ref{thm:isotropic}.  The self-case of
Theorem~\ref{thm:mixed-gaussian}, \eqref{eq:Jlower}, and
Theorem~\ref{thm:defect} give, with
\(\delta_k=\norm{T_{g_k}-\Id_4/4}_{\mathrm{HS}}\),
\[
 q(F)\geq\frac13M_{g_k}(F)^2
 \left(\frac16-2\delta_k^2\right).
\]
Letting \(k\to\infty\) gives the second inequality in
\eqref{eq:mass-sandwich}.  Combining the two sides gives
\[
 q(F)\geq\frac1{18}\cdot\frac83m(F)^2
 =\frac4{27}m(F)^2.
\]
\end{proof}

The proof also gives the useful implications
\[
 q(F)=0\Longrightarrow M_*(F)=0\Longrightarrow m(F)=0.
\]
The complementary-plane example realizes \(m(F)=0\) and \(M_*(F)=2\),
thereby separating the latter two invariants.

\section{Equality, stability, and small self-intersection}

The constant \(4/27\) combines three separately sharp numerical inputs---the
endpoint \(J(-1)=1/3\), the Hodge moment \(1/6\), and Rankin's
\(\gamma_{4,2}=3/2\)---with the pointwise projection estimate
\eqref{eq:projection}.  Simultaneous equality imposes an additional
compatibility problem.  Thus the proof establishes \(4/27\); determining the
optimal value remains open.

\begin{proposition}[Strictness at an attained positive minimizer]
\label{prop:strict-attained}
If \(m(F)>0\) and \(M_*(F)\) is attained by a volume-one metric, then
\[
 q(F)>\frac4{27}m(F)^2.
\]
\end{proposition}

\begin{proof}
At a minimizing metric, Theorem~\ref{thm:isotropic} gives \(T=\Id_4/4\),
and Theorem~\ref{thm:defect} gives \(\E B^2\geq1/6\).  Thus transverse
pairs have positive link-product measure.  On every such pair
\(x\cdot y>-1\), and \(J(x\cdot y)>J(-1)=1/3\).  The Gaussian formula
therefore gives \(q(F)>M_*(F)^2/18\).  Combining this strict inequality with
\eqref{eq:massprojection} proves the claim.
\end{proof}

Accordingly, equality in \eqref{eq:main} with \(m(F)>0\) forces nonattainment
of \(M_*(F)\), and every minimizing sequence leaves each compact subset of
\(\mathcal P\).  The preceding inequalities also give a precise test for near
extremizers.  Let \(F_k\) satisfy \(m(F_k)>0\) and
\[
 \frac{q(F_k)}{m(F_k)^2}\longrightarrow\frac4{27}.
\]
The mass sandwich forces
\[
 \frac{M_*(F_k)}{m(F_k)}\longrightarrow\sqrt{\frac83},
 \qquad
 \frac{q(F_k)}{M_*(F_k)^2}\longrightarrow\frac1{18}.
\]
Choose, by a diagonal application of Theorem~\ref{thm:isotropic}, metrics
\(g_k\) such that
\[
 \frac{M_{g_k}(F_k)}{M_*(F_k)}\longrightarrow1,
 \qquad
 \norm{T_k-\Id_4/4}_{\mathrm{HS},g_k}\longrightarrow0,
\]
and choose the Rankin quotient \(A_k\) used in the projection proof.  Write
\(P_k,Q_k\) for the remaining moments, \(\alpha_k\) for the primitive row
bivector, \(W_k\) for its dual plane, and \(\E_k\) for expectation at \(g_k\).
The chains of inequalities above then force:
\begin{enumerate}
\item \(\E_k B^2\to1/6\), hence
\[
 \norm{P_k-\Id_3/3}_{\mathrm{HS},g_k}^2+
 \norm{Q_k-\Id_3/3}_{\mathrm{HS},g_k}^2\longrightarrow0
\]
by \eqref{eq:defect-stability};
\item
\[
 \E_k\!\left[B^2\left(J(x\cdot y)-\frac13\right)\right]\to0,
\]
so the \(B^2\)-weighted transverse contribution concentrates toward the
antipodal endpoint of the Gaussian kernel;
\item the Rankin and projection steps become sharp in the integrated sense:
\[
 \norm{\alpha_k}_{g_k}^2\longrightarrow\frac32,qquad
 \frac{d_{A_k}(F_k)}{m(F_k)}\longrightarrow1,qquad
 \frac{d_{A_k}(F_k)}{
 \norm{\alpha_k}_{g_k}M_{g_k}(F_k)
 \tr(\Pi_{W_k}T_k)}\longrightarrow1.
\]
\end{enumerate}
These give necessary second-moment and concentration constraints for near
extremizers.  The full flag law retains additional freedom even when
both Hodge moments tend to \(\Id_3/3\).

For integral fans, \eqref{eq:main} yields the integer bound
\begin{equation}\label{eq:smallq}
 m(F)\leq\left\lfloor\frac{3\sqrt3}{2}\sqrt{q(F)}\right\rfloor.
\end{equation}
Thus \(q=0\) forces \(m=0\), \(q=1\) forces \(m\leq2\), and \(q=2\)
forces \(m\leq3\).  This restricts quotient degrees; classification of the
fan additionally requires hypotheses on realizability and on the chosen
quotient.

The case \(m(F)=0\) admits a direct description.  Because the fan is effective,
\[
 m(F)=0
\]
precisely when some primitive rational two-plane \(K\) meets every facet
plane of \(F\) in positive dimension: take \(K=\ker A\).  Effectivity makes
every full-rank push-forward contribution positive, so \(A_*F=0\) exactly
when \(A\) has rank at most one on every facet.
For \(m=1,2,3\), the collapsed fiber components of a low-degree quotient may
vary freely.  One may add effective balanced components that the same
quotient maps to lower dimension, and this operation preserves its degree.
Thus
transferring classification results from intrinsically degree-one tropical
linear spaces requires additional control of those collapsed components.

\section{Mixed intersection: a sufficient hypothesis and a sharp zero example}

For two nonzero fans, define the moments \(P_F,Q_F,R_F\) and
\(P_G,Q_G,R_G\) of their plane laws by \eqref{eq:PQR}.  The calculation leading to
\eqref{eq:Bmoment} gives
\begin{equation}\label{eq:mixed-moment}
 \E_{F,G}B^2=\frac14\left(
 \tr(P_FP_G)+\tr(Q_FQ_G)-2\ip{R_F}{R_G}\right).
\end{equation}
A mixed lower bound follows when one law is fully Pl\"ucker-isotropic.

\begin{corollary}[Mixed bound under Pl\"ucker isotropy]\label{cor:mixed}
Fix a common volume-one metric.  Suppose that \(F\neq0\) and its plane law is
Pl\"ucker-isotropic:
\[
 \E_F(z\otimes z)=\frac16\Id_{\wedge^2\R^4}.
\]
Equivalently,
\[
 P_F=Q_F=\frac13\Id_3,\qquad R_F=0.
\]
Then, for every effective balanced rational two-fan \(G\) in \(N_\R\),
\begin{equation}\label{eq:mixed-bound}
 \deg(F\cdot G)\geq\frac1{18}M_g(F)M_g(G)
 \geq\frac1{18}M_*(F)M_*(G)
 \geq\frac4{27}m(F)m(G).
\end{equation}
\end{corollary}

\begin{proof}
If \(G=0\), every term in \eqref{eq:mixed-bound} vanishes.  Assume
\(G\neq0\).  Equation \eqref{eq:mixed-moment} then gives
\(\E_{F,G}B^2=1/6\), independently
of the law of \(G\).  Apply the mixed Gaussian formula and \(J\geq1/3\),
then the mass comparison \eqref{eq:massprojection} to both fans.
\end{proof}

More generally, for nonzero \(F,G\), if \(R_F=0\), then
\[
 \E_{F,G}B^2\geq\frac14\left(
 \lambda_{\min}(P_F)+\lambda_{\min}(Q_F)\right).
\]
Thus a uniform lower bound for
\(\lambda_{\min}(P_F)+\lambda_{\min}(Q_F)\) gives a mixed-intersection bound
through Theorem~\ref{thm:mixed-gaussian}.

\begin{example}[Common isotropy with zero mixed intersection]
\label{ex:mixed-obstruction}
In the standard lattice set
\[
 F=[P_{12}]+[P_{34}],\qquad
 G=[P_{13}]+[P_{24}],
\quad P_{ij}=\Span(e_i,e_j).
\]
Every \(F\)-plane meets every \(G\)-plane in a line, so
\[
 \deg(F\cdot G)=0.
\]
At the same time \(M_*(F)=M_*(G)=2\).  For example,
\[
 \norm{e_1\wedge e_2}\norm{e_3\wedge e_4}\geq
 \norm{e_1\wedge e_2\wedge e_3\wedge e_4}=1,
\]
so the arithmetic--geometric mean inequality gives \(M_g(F)\geq2\);
the standard metric realizes equality, and likewise for \(G\).  That metric
also has
\[
 T_F=T_G=\frac14\Id_4.
\]
Hence simultaneous mass minimization and ambient isotropy permit zero mixed
intersection.  Since mixed degrees are nonnegative, the largest constant
valid in
\[
 \deg(F\cdot G)\geq C M_*(F)M_*(G)
\]
for all pairs is \(C=0\).  Here \(m(F)=m(G)=0\), so the example leaves a
bound formulated with \(m(F)m(G)\) vacuous.
\end{example}

\section{Intersection rings, Hodge theory, and matroidal fans}

If \(K_A=\ker(A)\) is given its primitive complete-plane weight, then the
projection degree has the intersection-theoretic interpretation
\begin{equation}\label{eq:kernel-pairing}
 d_A(F)=\deg(F\cdot[K_A]).
\end{equation}
Indeed, the local index in a generic intersection with a translate of
\(K_A\) is the same quotient-lattice index that appears in the tropical
push-forward.  Thus \(q(F)\) and every \(d_A(F)\) are intersection numbers in
the stable-intersection ring; the affine mass \(M_*(F)\) mediates the
comparison between \(q(F)\) and \(m(F)=\min_A d_A(F)\).

Fulton and Sturmfels identify products of toric Minkowski weights by a
lattice displacement rule \cite{FultonSturmfels}; Jensen and Yu identify the
rational tropical fan-cycle ring with McMullen's polytope algebra
\cite{JensenYu,McMullen}.  These identifications encode the products
algebraically.  The quantitative estimate uses positivity of the spherical
link measure together with affine minimization.  These inputs supplement the
algebraic descriptions above: \([K_A]\) is a codimension-two linear cycle,
standard Hodge-index inequalities concern ample divisor classes, and the
rational polytope-algebra decomposition allows coefficients of both signs.

Amini and Piquerez prove Hodge--Riemann relations and Hard Lefschetz for the
classes of tropical and matroidal fans treated in \cite{AminiPiquerez}.
Wood extends the
Maxim--Sch\"urmann intersection-cohomology signature formula to arbitrary
convex polytopes in the BBFK framework and discusses a corresponding
Hodge-index statement \cite{Wood}.  Within their respective settings, these
results determine signatures and establish nondegeneracy properties of the
intersection pairing.  Affine normalization and Rankin's theorem supply the
arithmetic ingredient in comparing the minimum of
\(\deg(F\cdot[K_A])\) over primitive kernel planes with \(q(F)\).

For cycles contained in a Bergman fan, Shaw constructs an intersection
product using matroidal modifications \cite{ShawMatroidal}.  For a chosen
integral embedding, \(M_g\) and \(T\) are finite sums of cone-wise spherical
integrals, while \(P,Q,R\) are weighted sums over facet planes; all may depend
on the embedding and metric.  Matroid invariance would amount to proving
independence of these choices.

\begin{example}[The standard tropical plane]\label{ex:U35}
Consider the fine Bergman fan \(B(U_{3,5})\) in
\[
 N_\R=\R^5/\R(1,1,1,1,1).
\]
Let \(u_i\) be the image of the \(i\)-th standard basis vector.
Its twenty maximal cones are
\[
 \operatorname{cone}(u_i,u_i+u_j),\qquad i\neq j,
\]
with their canonical tropical unit weights
\cite{ArdilaKlivans,SpeyerTropicalLinearSpaces}.  In the quotient Euclidean
metric,
\[
 \norm{u_i\wedge u_j}=\sqrt{\frac35},\qquad
 \cos\angle(u_i,u_i+u_j)=\sqrt{\frac38}.
\]
The quotient lattice has covolume \(1/\sqrt5\).  After scaling the metric
to lattice covolume one, every facet bivector has norm
\(\sqrt3/5^{1/4}\), and the scaling preserves the angle.  Hence
\begin{equation}\label{eq:U35mass}
 M_g(B(U_{3,5}))=
 \frac{10\sqrt3}{\pi5^{1/4}}
 \arccos\sqrt{\frac38}.
\end{equation}
The alternating group \(A_5\) acts through its irreducible standard real
representation on \(\mathbf1^\perp\), so \(T=\Id_4/4\).  Moreover,
\(\bigwedge^2\mathbf1^\perp\simeq\mathbf3\oplus\mathbf3'\); these two
inequivalent irreducibles are the self-dual and anti-self-dual spaces.
Schur's lemma gives
\[
 P=Q=\frac13\Id_3,\qquad R=0.
\]
Thus this link law is Pl\"ucker-isotropic, \(\E B^2=1/6\), and the standard
volume-one metric is critical for \(M_g\).  Equation~\eqref{eq:U35mass} gives
its mass there; global minimality requires a separate argument.
\end{example}

The theorem covers both realizable and nonrealizable balanced fans.  Brugall\'e
and Shaw exhibit additional local intersection obstructions for curves
realizable in tropical surfaces \cite{BrugalleShaw}.

\section{Higher-dimensional formulas and isotropic zero examples}

The coarea identity and the affine first variation extend to every
dimension.  Let \(F,G\) be effective balanced rational \(d\)-fans in a
rank-\(2d\) lattice, represented using a fixed pointed rational subdivision.
Write
\[
 q_d(F)=\deg(F\cdot F),\qquad
 m_d(F)=\min_{A:N\twoheadrightarrow\Z^d}\deg(A_*F),
\]
and set
\[
 \omega_{d-1}=\operatorname{area}(S^{d-1}).
\]
On the spherical link of a facet, define
\begin{equation}\label{eq:highmass}
 d\mu_{F,g}=\frac{w_\sigma\norm{p_\sigma}_g}{\omega_{d-1}}\,
 d\mathcal H^{d-1}.
\end{equation}
Put
\[
 M_g(F)=\mu_{F,g}(S_g^{2d-1}),\qquad
 M_*(F)=\inf_{\covol_g(N)=1}M_g(F).
\]
For \(F\neq0\), division by \(M_g(F)\) defines its probability law on
oriented flags; \(\E_{F,G}\) denotes the product of the laws of two nonzero
fans.
At a regular point choose an oriented orthonormal frame
\[
 z=x\wedge t_1\wedge\cdots\wedge t_{d-1}.
\]
For two such frames let \(B(z,z')\) be their \(2d\)-dimensional
determinant, and put
\begin{equation}\label{eq:Jd}
 \mathcal J_d(c)=\frac{\omega_{d-1}^2}{(2\pi)^d}
 \int_{r,\rho>0}r^{d-1}\rho^{d-1}
 e^{-(r^2+\rho^2-2cr\rho)/2}\,dr\,d\rho,
 \qquad -1\leq c<1.
\end{equation}

\begin{theorem}[Higher-dimensional Gaussian formula]\label{thm:high-gaussian}
For \(F\neq0\), \(G\neq0\), and a volume-one metric,
define the entire kernel to be zero on a facet pair whose planes span a
proper subspace of \(\R^{2d}\).  Then
\begin{equation}\label{eq:high-gaussian}
 \deg(F\cdot G)=M_g(F)M_g(G)
 \E_{F,G}\!\left[B(z,z')^2\mathcal J_d(x\cdot y)\right].
\end{equation}
For \(F=0\) or \(G=0\), equation
\eqref{eq:high-gaussian} carries the convention that both sides equal zero.
For \(-1\leq c<1\), the function \(\mathcal J_d\) is increasing and
\begin{equation}\label{eq:Jdmin}
 \mathcal J_d(c)\geq\mathcal J_d(-1)
 =\frac{2\Gamma(d)^3}
 {\Gamma(2d)\Gamma(d/2)^2}.
\end{equation}
\end{theorem}

\begin{proof}
The parameterization of a difference cone has Jacobian
\[
 r^{d-1}\rho^{d-1}|B(z,z')|,
\]
and the transverse lattice index contributes
\(\norm{p_\sigma}\norm{p_\tau}|B(z,z')|\).  As in the proof of
Theorem~\ref{thm:mixed-gaussian}, conical scaling permits averaging the
fan-displacement rule over an unrestricted Gaussian.  Comparing the
resulting facet-pair sum with \eqref{eq:highmass} proves
\eqref{eq:high-gaussian}.  The integrand is
pointwise increasing in \(c\), through the factor \(e^{cr\rho}\).  At
\(c=-1\), the substitution
\(s=r+\rho\), \(t=r/s\) gives
\[
 \int r^{d-1}\rho^{d-1}e^{-(r+\rho)^2/2}\,dr\,d\rho
 =\frac{2^{d-1}\Gamma(d)^3}{\Gamma(2d)}.
\]
Multiplication by
\(\omega_{d-1}^2/(2\pi)^d=
2^{2-d}/\Gamma(d/2)^2\) proves \eqref{eq:Jdmin}.
\end{proof}

Under full Pl\"ucker isotropy, the mixed second moment is fixed.  If
\[
 \E_F(z\otimes z)=
 \binom{2d}{d}^{-1}\Id_{\wedge^d\R^{2d}},
\]
then \(\E_{F,G}B^2=\binom{2d}{d}^{-1}\) for every nonzero \(G\), and
\begin{equation}\label{eq:high-mixed}
 \deg(F\cdot G)\geq
 \frac{\mathcal J_d(-1)}{\binom{2d}{d}}M_g(F)M_g(G).
\end{equation}
The inequality also holds for \(G=0\), when both sides vanish.

The first-variation argument has the following \(d\)-dimensional form.  For a
stationary weighted
\((d-1)\)-complex in \(S^{2d-1}\), set
\[
S_g=\E(\Pi_z),\qquad A_g=\E(xx^T).
\]
Stationarity gives \(S_g=dA_g\).  Indeed, for the
\((d-1)\)-dimensional tangent projection \(\Pi_{\mathrm{tan}}\), integration of
\[
\Delta_{\rm link}(x_i x_j)
 =2(\Pi_{\mathrm{tan}})_{ij}-2(d-1)x_ix_j
\]
has its boundary contributions canceled by stationarity, and hence
\(\E\Pi_{\mathrm{tan}}=(d-1)\E(xx^T)\); adding the radial projection proves the
claim.  The link of a balanced rational fan is stationary with precisely the
density in \eqref{eq:highmass}.  Indeed, along a codimension-one face
\(\tau\), choose a primitive multivector \(p_\tau\) of its face lattice and
an outward lift \(u_\sigma\) of each primitive quotient normal, so that
\(p_\sigma=p_\tau\wedge u_\sigma\).  If \(n_\sigma\) is the corresponding
outward unit conormal, then
\[
 \norm{p_\sigma}n_\sigma
 =\norm{p_\tau}\operatorname{pr}_{(\Span\tau)^\perp}(u_\sigma),
\]
so the conormal boundary terms cancel by tropical balancing.  Under
\(h_\eta=e^{\eta H}\), the density
in \eqref{eq:highmass} is multiplied by
\[
 \frac{\norm{\bigwedge^dh_\eta z}^{\,2}}{\norm{h_\eta x}^{\,d}},
\]
and hence
\[
 DM_g(H)=M_g\E\!\left[
 2\tr(H\Pi_z)-d\ip{Hx}{x}\right]
 =M_g\tr(HS_g).
\]
On \(\mathcal P_d=SL_{2d}(\R)/SO(2d)\), use the analogous invariant metric
normalization.  The last display says
\[
 \nabla\log M_g=S_g-\frac12\Id_{2d}.
\]
Since \(\mathcal P_d\) is a proper Hadamard manifold and the same
fixed-subdivision argument makes \(M_g(F)\) smooth,
Lemma~\ref{lem:approx-critical} produces, whenever
\(M_*(F)>0\),
\begin{equation}\label{eq:high-isotropy}
 M_{g_k}(F)\to M_*(F),\qquad
 \norm{S_{g_k}-\frac12\Id_{2d}}_{\mathrm{HS},g_k}\to0.
\end{equation}

Equation~\eqref{eq:high-isotropy} controls the first moment \(S_g\).  In
dimension three this still permits \(\E B^2=0\), as the following example
shows.

\begin{example}[Isotropic zero-intersection example in dimension three]
\label{ex:d3}
In \(\R^6\), let
\[
 H=[P_{123}]+[P_{145}]+[P_{246}]+[P_{356}].
\]
Every pair of these coordinate three-planes has a common coordinate line,
so
\[
 \deg(H\cdot H)=0,\qquad \E B^2=0.
\]
Every coordinate direction occurs in exactly two of the four planes, hence
the standard metric satisfies \(\E\Pi_z=\Id_6/2\).

This metric is globally mass minimizing.  Let \(G\) be the Gram matrix of
an arbitrary volume-one metric and let \(S_1,\ldots,S_4\) be the four
coordinate triples above.  Since each coordinate occurs in two of these sets,
the fractional Hadamard inequality \cite{MadimanTetali} gives
\[
 \det G\leq\prod_{i=1}^4\det G[S_i]^{1/2}.
\]
Since \(\det G=1\),
\[
 \prod_i\norm{p_{S_i}}\geq1,\qquad
 \sum_i\norm{p_{S_i}}\geq4.
\]
Thus, in the normalization \eqref{eq:highmass},
\[
 M_*(H)=4>0.
\]
Finally \(m_3(H)=0\): the kernel \(P_{123}\) meets every component plane,
so all their images have dimension at most two.
\end{example}

This forces \(C=0\) in every bound
\(q_3(F)\geq C M_*(F)^2\) based only on ambient isotropy.  Since
\(m_3(H)=0\), the corresponding \(m_3\)-inequality is vacuous on this
example.

\subsection{The higher-dimensional projection step}

For a rank-\(d\) quotient there is an exact angular expression.  Let
\(A:N\twoheadrightarrow\Z^d\), and let \(\alpha\in\bigwedge^dN^*\) be its
primitive row multivector.  With a volume-one target metric, put
\[
 D_A(z)=|\det(Ax,At_1,\ldots,At_{d-1})|.
\]
On a full-rank facet,
\[
 D_A(z)=\frac{|\alpha(p_\sigma)|}{\norm{p_\sigma}},\qquad
 \operatorname{Jac}_{d-1}\!\left(x\longmapsto
 \frac{Ax}{\norm{Ax}}\right)
 =\frac{D_A(z)}{\norm{Ax}^{d}}.
\]
The second identity follows by comparing the radial Jacobian of
\((r,x)\mapsto rAx\) with polar coordinates in the target.  Angular
averaging of the push-forward fan therefore gives
\[
 d_A(F)
 =\frac1{\omega_{d-1}}
 \sum_{\rank(A|_{\Span\sigma})=d}w_\sigma|\alpha(p_\sigma)|
 \int_{\sigma\cap S_g^{2d-1}}
 \frac{D_A(z)}{\norm{Ax}^{d}}\,d\mathcal H^{d-1}.
\]
Using the first identity and \eqref{eq:highmass}, this becomes
\begin{equation}\label{eq:high-projection-exact}
 d_A(F)=\int
 \frac{D_A(z)^2}{\norm{Ax}^{\,d}}\,d\mu_{F,g}.
\end{equation}
On a rank-deficient facet the integrand is assigned value zero, matching its
zero top-dimensional push-forward contribution.
For \(d=2\), the pointwise inequality
\[
 \frac{D_A(z)^2}{\norm{Ax}^2}\leq\norm{At}^2
\]
turns \eqref{eq:high-projection-exact} into the linear trace estimate
\eqref{eq:projection}.  For \(d\geq3\), suppose that
\(\norm{Ax}=\varepsilon\) and the other projected frame vectors have unit
size.  Then the integrand can be of order \(\varepsilon^{2-d}\), which is
unbounded as \(\varepsilon\to0\).  A higher-dimensional estimate must
therefore use the integrated geometry of the stationary link.

The obstruction separates into two quantitative hypotheses.  Let
\[
 \gamma_{2d,d}=
 \sup_{\covol(L)=1}
 \min_{\substack{\alpha\in\wedge^dL\\
 \alpha\ {\rm primitive,\ decomposable}}}\norm\alpha^2
\]
where the supremum runs over rank-\(2d\) Euclidean lattices \(L\).  This is
the Hermite--Rankin constant.  Assume there are constants
\(\beta_d,\rho_d>0\), independent of \(F\), such that every fan \(F\) in a
given class with \(M_*(F)>0\) has an isotropic almost-minimizing sequence
\(g_k\) for which
\[
 \liminf_{k\to\infty}\E_{g_k} B^2\geq\beta_d.
\]
Suppose that, along the same sequence, one may choose for each \(k\) a Rankin
quotient \(A_k\), with primitive row multivector \(\alpha_k\), satisfying
\[
 \norm{\alpha_k}_{g_k}^2\leq\gamma_{2d,d},\qquad
 d_{A_k}(F)\leq(\rho_d+o_{k\to\infty}(1))
 \norm{\alpha_k}_{g_k}M_{g_k}(F).
\]
For every such \(F\), the projection hypothesis gives
\(m_d(F)\leq\rho_d\sqrt{\gamma_{2d,d}}M_*(F)\).  The Gaussian formula and
the transversality hypothesis give
\(q_d(F)\geq\mathcal J_d(-1)\beta_dM_*(F)^2\).  Hence
\begin{equation}\label{eq:conditional-high}
 q_d(F)\geq
 \frac{\mathcal J_d(-1)\beta_d}
 {\rho_d^2\gamma_{2d,d}}\,m_d(F)^2.
\end{equation}
For \(d=2\),
\[
 \mathcal J_2(-1)=\frac13,\quad
 \beta_2=\frac16,\quad
 \rho_2=\frac12,\quad
 \gamma_{4,2}=\frac32,
\]
and the coefficient in \eqref{eq:conditional-high} equals \(4/27\).
In this case the coarse estimate \eqref{eq:coarse} also covers
\(M_*(F)=0\).
For the class of all three-fans, Example~\ref{ex:d3} shows that the largest
admissible value of \(\beta_3\) is zero.  Determining the optimal projection
constant \(\rho_d\) in the preceding asymptotic projection hypothesis remains
open.

\section{Conclusion}

The proof factors through
\[
 m(F)\leq\sqrt{\frac38}\,M_*(F),\qquad
 q(F)\geq\frac1{18}M_*(F)^2.
\]
Affine normalization supplies the common almost-isotropic sequence.  The
projection formula and Rankin's theorem give the first estimate; the Gaussian
formula and the four-dimensional Hodge decomposition give the second.  Three
questions remain: characterize the fans with
\(M_*=0\); find hypotheses under which ambient isotropy yields positive
Pl\"ucker transversality in higher dimensions; and determine the optimal
constant in the four-dimensional inequality.

\section*{Data availability}
No new data were generated or analysed in support of this research.

\section*{Acknowledgements}
ChatGPT Pro (GPT-5.6) was used to assist in finding and subsequently editing a
proof of the main estimate.  OpenAI Codex was used in preparing this manuscript
for language editing, \LaTeX{} preparation, literature and bibliography checks,
and assistance in drafting, restructuring, and checking portions of the
exposition and mathematical arguments.  The author assumes full responsibility
for the final text, claims, proofs, and references.

\sloppy

\fussy

\end{document}